\documentclass[11pt,a4paper,oneside,reqno]{amsart}
\usepackage[DIV=11]{typearea}
\usepackage[T1]{fontenc}
\usepackage{amsmath,amsthm,amssymb,amsfonts,float}
\usepackage[english]{babel}
\usepackage{tikz,float} 
\usetikzlibrary{patterns}
\newcommand{\dist}{\operatorname{dist}}
\newcommand{\diam}{\operatorname{diam}}
\newcommand{\ran}{\operatorname{Ran}}
\newcommand{\ceil}[1]{\lceil #1 \rceil}
\newcommand{\NN}{\mathbb{N}}
\newcommand{\RR}{\mathbb{R}}
\newcommand{\ZZ}{\mathbb{Z}}
\newtheorem{thm}{Theorem}
\theoremstyle{remark}
\newtheorem{rem}[thm]{Remark}
\newcommand{\hm}[1]{\leavevmode{\marginpar{\tiny%
 $ \hbox to 0mm{\hspace*{-0.5mm} $ \leftarrow $ \hss}%
 \vcenter{\vrule depth 0.1mm height 0.1mm width \the\marginparwidth}%
 \hbox to
 0mm{\hss $ \rightarrow $ \hspace*{-0.5mm}} $ \\\relax\raggedright #1}}}
\begin{document}
\title[Sampling inequality for $L^2$-norms of eigenfunctions]{Sampling inequality for $L^2$-norms of eigenfunctions, spectral projectors, and Weyl sequences of  Schr\"odinger operators}
\author[M.~Tautenhahn]{Martin Tautenhahn}
\address{Technische Universit\"at Chemnitz,
Fakult\"at f\"ur Mathematik, Germany}
 \email{martin.tautenhahn@mathematik.tu-chemnitz.de}
\author[I.~Veseli\'c]{Ivan Veseli\'c}

\begin{abstract}
 
We consider a Schr\"odinger operator with bounded, measurable potential in multidimensional Euclidean space.
We prove for every $L^2$-eigenfunction a quantitative equidistribution estimate.
It compares the total $L^2$-norm with the $L^2$-norm over an equidistributed collection of balls.
Our estimate is explicit with respect to the radius of the balls, norm of the potential and the
energy of the eigenfunction. Similar estimates also hold for Weyl sequences and for linear combinations of eigenfunctions,
as long as the associated eigenvalues are sufficiently close.

\end{abstract}


\maketitle

Let $M,\delta > 0$. We say that a sequence $Z = (z_k)_{k \in (M\ZZ)^d}\subset \RR^d$ 
is \emph{$(M,\delta)$-equidistributed}, if for all $k \in (M \ZZ)^d$ we have
$
  B(z_k , \delta) := \{x \in \RR^d \colon \lvert x-z_k \rvert_2 < \delta\} \subset \Lambda_M (k) := 
k + (-M/2 , M/2)^d .
$ 
For a $(M,\delta)$-equidistributed sequence $Z$ we set $S_{\delta, Z} = \cup_{k \in (M \ZZ)^d } B(z_k , \delta)$
and $W_{\delta, Z}:=\chi_{S_{\delta, Z}} $. Note that every $(M,\delta)$-equidistributed set $Z$ is relatively dense (as used in the context of Delone sets).
Conversely, for every relatively dense set $X$ and $\delta>0$ we can find a scale $M$  and subset $Z \subset X$ such that $Z$ is 
$(M,\delta)$-equidistributed.
\begin{figure}[ht]\centering
\begin{tikzpicture}
\pgfmathsetseed{{\number\pdfrandomseed}}
\foreach \x in {0.5,1.5,...,4.5}{
  \foreach \y in {0.5,1.5,...,4.5}{
    \filldraw[fill=gray!70] (\x+rand*0.35,\y+rand*0.35) circle (0.15cm);
  }
}
\foreach \y in {1,2,3,4}{
\draw (\y,0) --(\y,5);
\draw (0,\y) --(5,\y);
}

\begin{scope}[xshift=-6cm]
\foreach \x in {0.5,1.5,...,4.5}{
  \foreach \y in {0.5,1.5,...,4.5}{
    \filldraw[fill=gray!70] (\x,\y) circle (1.5mm);
  }
}
\foreach \y in {1,2,3,4}{
  \draw (\y,0) --(\y,5);
  \draw (0,\y) --(5,\y);
}
\end{scope}
\end{tikzpicture}
\caption{Illustration of $S_{\delta,Z} \subset \mathbb{R}^2$ for periodically (left) and non-periodically (right) arranged $Z$.\label{fig:equidistributed}}
\end{figure}
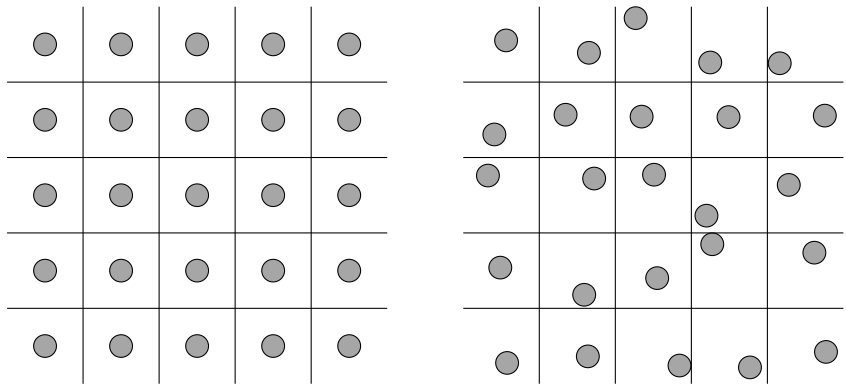

\begin{thm}
\label{t:scale_free-UCP}
There exists a constant $K_0\in(0,\infty)$ depending merely on the dimension $d$, such that for any $M >0$, $\delta \in (0,M/2]$, any $(M,\delta)$-equidistributed sequence 
$Z$,
any measurable and bounded $V\colon {\RR^d}\to \RR$, and any real-valued $\varphi\in W^{2,2}(\RR^d)$ satisfying
$\lvert \Delta \varphi \rvert \leq \lvert (V-E)\varphi \rvert $ a.\ e.\ on $\RR^d$ we have
\begin{equation*}
\label{eq:M-scale_free-UCP}
\left( \frac{\delta}{M}\right)^{K_0(1+M^{4/3} \lVert V-E \rVert_\infty^{2/3})}
\lVert \varphi \rVert 
\leq \lVert W_{\delta, Z}\varphi \rVert 
\leq \lVert \varphi \rVert .
\end{equation*}
\end{thm}
This is a corollary of the proof of the main result of \cite{RojasMolinaV-13}, where Schr\"odinger operators on a sequence of boxes 
$\Lambda_L(0)$, $L \in \NN$, are considered. J.~M.~Barbaroux asked the question whether \cite{RojasMolinaV-13} holds for operator on the whole space as well.
Let us sketch which modifications are necessary in the proof of \cite[Theorem~2.1]{RojasMolinaV-13}
to obtain the above result.
\begin{proof}
(1) By scaling it suffices to consider $M=1$.
(2) One does not need to extend $\varphi$ further, since it is already defined on the whole of $\RR^d$.
(3) A site $k \in \ZZ^d$ is called dominant if 
\begin{equation} \label{eq:dominant}
 \int_{\Lambda_1 (k)} \lvert \varphi \rvert^2 \geq \frac{1}{2T^d} \int_{\Lambda_T (k)} \lvert \varphi \rvert^2
\end{equation}
with $T= 62\ceil{\sqrt{d}}$, and otherwise weak. This corresponds to the notion chosen in \cite{RojasMolinaV-13}
in the case of periodic boundary conditions.
(4) Estimating the contribution of boxes centered at weak sites --- now the sum is over the infinite set $ \ZZ^d$, but all sums are finite since $\varphi \in L^2(\RR^d)$ --- 
gives again 
$\lVert \varphi\rVert^2 < 2 \lVert \chi_D \varphi\rVert^2$, where $D$ denotes the union of those boxes $ \Lambda_1(k)$ such that $k \in \ZZ^d$ is dominant.
(5) 
 For a dominating site $k \in \ZZ^d$ we define the right near-neighbor by $k^+ := k + (\lceil \sqrt{d} \rceil + 1) \mathbf{e_1}$ 
and  set $R:=R_k:=\lceil \sqrt{d} \rceil+y_k$ with 
$y_k := \langle e_1 , z_{k^+} \rangle - \langle e_1 , k^+ \rangle + 1/2 \in [0,1]$. Then 
$\Theta :=\Lambda_1(k)$ is disjoint from the open ball $B(z_{k^+},R)$. On the other hand there exists
an $a \in \Lambda_1(k)$ with $\lvert a-z_{k^+} \rvert_2=R$. Thus for any $b \in \Lambda_1(k)$ we have
$\lvert b-z_{k^+}\rvert_2\leq \lvert b-a\rvert_2 + \lvert a-z_{k^+}\rvert_2\leq \sqrt{d}+ R\leq 2R$. Thus 
$\Theta \subset \overline{B (z_{k^+},2R)} \setminus B (z_{k^+},R)$.
(6) Once this geometric condition is satisfied, the proof of Corollary 3.2 in 
\cite{RojasMolinaV-13} applies.
(7)
Hence  for every dominating site $k\in \ZZ^d$ we have
\[
\lVert \chi_{B(z_{k^+},\delta)} \varphi \rVert^2 \geq C_{
\mathrm{qUC}} \lVert \chi_{\Lambda_1(k)} \varphi \rVert^2
\]
with the constant $C_{
\mathrm{qUC}}$ arising from the quantitative unique continuation estimate, i.e.\ Corollary~3.2 in \cite{RojasMolinaV-13}. 
(8)
Taking the sum over all dominating sites $k \in \ZZ^d$ we obtain
\begin{equation*}
\label{eq:sum}
\sum_{k \in \ZZ^d \text{ dominating}} \lVert \chi_{B(z_{k^+},\delta)} \varphi \rVert^2
\geq \frac{C_{
\mathrm{qUC}}}{2}\lVert \varphi \rVert^2 .
\end{equation*}
 The result follows by using the quantitative estimate of $C_{\mathrm{qUC}}$ provided in \cite{RojasMolinaV-13}.
\end{proof}

\begin{rem}
We  use the opportunity to correct a minor mistake in 
\cite{RojasMolinaV-13}. There, in the statement of Theorem 3.1 and Corollary 3.2 the geometric condition
$\diam \Theta \leq  R =\dist (x,\Theta)$ should be replaced by
$\Theta \subset \overline{B (x,2R)} \setminus B (x,R)$. The latter is the property actually used in the proof. This property is satisfied in the application, i.e.\ the proof of \cite[Theorem~2.1]{RojasMolinaV-13}, as explained in Step (5) above.
\end{rem}
\begin{rem}
While for Schr\"odinger operators on a box eigenfunctions capture the whole spectrum, for analogs on $\RR^d$
continuous spectrum may exist as well. Thus attention cannot be restricted to eigenfunctions only, as pointed out in a discussion by G.~Raikov.
Hence studying unique continuation of spectral projectors 
(rather than of eigenfunctions) is here even more important than for operators on cubes, 
as done in \cite{Klein-13} and \cite{NakicTTV-15}. Nevertheless there are important classes of potentials 
(short range or random, for instance) where a substantial part of the spectrum consists of eigenvalues. 
C.~Rojas Molina pointed out that in view of the close relationship between proofs in \cite{RojasMolinaV-13} and 
\cite{Klein-13}, the latter should have an infinite volume analogue like Theorem \ref{t:scale_free-UCP} above as well.  
Indeed we have:
\end{rem}
\begin{thm} \label{t:scale_free-UCP2}
There exists a constant $K_1 \geq 1$ depending merely on the dimension $d$, 
such that for all $E_0,M > 0$, all $\delta \in (0,M/2)$, all $(M,\delta)$-equidistributed sequences $Z$, 
all measurable and bounded $V\colon {\RR^d}\to \RR$, and all intervals $I \subset (-\infty , E_0]$ with 
\[
 \vert I \rvert \leq 2 \gamma \quad \text{where} \quad
 \gamma^2 = \frac{1}{2M^4} \left(\frac{\delta}{M}\right)^{K_1 \bigl(1+ M^{4/3}(2\lVert V\rVert_\infty + E_0)^{2/3} \bigr)},
\]
we have
 \[
  \chi_I (H) W_{\delta, Z} \chi_I (H) \geq M^4\gamma^2 \chi_I (H)  \quad \text{where} \quad H:= -\Delta +V. 
 \]
 \end{thm}
For the proof of Theorem~\ref{t:scale_free-UCP2} we shall need the following quantitative unique continuation principle.
\begin{thm}\label{thm:mit-Rest}
There exists $K_1 \geq 1$ depending only on the dimension, such that for all $M>0$, all $\delta \in (0,M/2)$, all $(M,\delta)$-equidistributed sequences 
$Z$ and all real-valued $\psi \in W^{2,2} (\RR^d)$ we have
\[
\left(\frac{\delta}{M} \right)^{K_1 (1+ M^{4/3}\lVert V\rVert_\infty^{2/3})}
 \lVert \psi \rVert^2
 \leq
 \lVert W_{\delta, Z} \, \psi \rVert^2 + \delta^2 M^2 \lVert -\Delta \psi + V\psi \rVert^2 .
 \]
\end{thm}
\begin{proof}
By scaling it sufficies to consider $M = 1$. We follow \cite[Proof of Theorem 2.2]{Klein-13} and fix $\psi \in W^{2,2} (\RR^d)$. A site $k \in \ZZ^d$ is called dominant if Ineq.~\eqref{eq:dominant} holds with $T = 46\sqrt{d}$. We denote by $D$ the union of those boxes $ \Lambda_1(k)$ such that $k \in \ZZ^d$ is dominant. 
Then $\lVert \psi\rVert^2 < 2 \lVert \chi_D \psi\rVert^2$. 
For a dominating site $k \in \ZZ^d$ we define the right near neighbor by $k^+ = k + 2\mathbf{e_1}$
and want to apply \cite[Theorem 2.1]{Klein-13}, see also \cite{BourgainK-13}, with
\[
 \Omega = \Lambda_T (k),\quad
 \Theta = \Lambda_1 (k), \quad
 x_0 = z_{k^+} .
\]
Thus, we have to check the assumptions of \cite[Theorem 2.1]{Klein-13}. 
We obviously have $Q = Q (x_0,\Theta) \geq 1$ and $Q = Q (x_0 , \Theta) \leq 3\sqrt{d}$.
({Note that the inequality $Q = Q (x_0 , \Theta) \leq (5/2)\sqrt{d}$ as written in \cite{Klein-13} is not true, e.g.\ if $d=1$.})
Hence, for all $y \in B (z_{k^+} , 6Q + 2)$
\begin{align*}
 \lvert k - y \lvert &\leq \lvert k-k^+  \rvert + \lvert k^+ - z_{k^+} \rvert + \lvert z_{k^+} - y  \rvert 
  < 2 + \sqrt{d} + 6Q+2 
  \leq 23\sqrt{d} = T / 2 ,
\end{align*}
and hence $B (z_{k^+} , 6Q + 2) \subset B (k , T/2) \subset \Lambda_T (k)$.
Thus, Theorem 2.1 of \cite{Klein-13} gives us for all dominating sites $k \in \ZZ^d$
\[
 \left( \frac{\delta}{Q} \right)^{K_1' (1+\lVert V\rVert_\infty^{2/3})}
 \lVert \chi_{\Lambda_1 (k)} \psi \rVert^2
 \leq
 \lVert \chi_{B(z_{k^+} , \delta)} \psi \rVert^2 + \delta^2 \lVert \chi_{\Lambda_T (k)} ( -\Delta \psi + V\psi ) \rVert^2 ,
\]
where $K_1' \geq 1$ is some constant depending only on the dimension. Now, summing over dominating $k \in \ZZ^d$ we obtain
\[
\left( \frac{\delta}{Q} \right)^{K_1' (1+\lVert V\rVert_\infty^{2/3})}
 \lVert \chi_D \psi \rVert^2
 \leq
 \lVert W_{\delta, Z} \, \psi \rVert^2 + \lceil T \rceil^d \delta^2 \bigl\lVert -\Delta \psi + V\psi \rVert^2 .
\]
Using $\lVert \varphi\rVert^2 < 2 \lVert \chi_D \varphi\rVert^2$ we obtain the statement of the theorem.
\end{proof}
We are now ready to prove Theorem~\ref{t:scale_free-UCP2}.
\begin{proof}[Proof of Theorem~\ref{t:scale_free-UCP2}]
We follow \cite[Proof of Theorem 1.1]{Klein-13}. Without loss of generality we assume that $I = [E-\gamma,E+\gamma]$ with some $E \in [-\lVert V\rVert_\infty  , E_0 ]$. Hence,
 \[
  \lVert V-E \rVert_\infty \leq \lVert V\rVert_\infty + \lvert E \rvert \leq 2\lVert V\rVert_\infty + E_0 .
 \]
If $\psi \in \ran \chi_I (H)$ is real-valued, we apply Theorem~\ref{thm:mit-Rest} with $V$ replaced by $V-E$ to $\psi$, 
and obtain, by using $\lVert (H-E) \psi \rVert^2 \leq \gamma^2 \lVert \psi \rVert^2$, the inequality
\begin{equation*}
 2 M^4 \gamma^2 \lVert \psi \rVert^2
 \leq
 \lVert W_{\delta, Z} \, \psi \rVert^2
 +
 \delta^2 M^2 \lVert -\Delta \psi + V \psi - E \psi \rVert^2 
  \leq 
 \lVert W_{\delta, Z} \, \psi \rVert^2
 +
 \delta^2 M^2 \gamma^2 \lVert \psi \rVert^2 .
\end{equation*}
Using $2M^4\gamma^2 - \delta^2 M^2 \gamma^2 \geq \gamma^2 M^4$, this proves the theorem if $\psi$ is real-valued. If $\psi$ is complex-valued we use $\lVert \psi \rVert^2 = \lVert \Re \psi \rVert^2 + \lVert \Im \psi \rVert^2$ and $\Re \psi, \Im \psi \in \ran (\chi_I (H))$, and conclude the statement of the theorem.
\end{proof}

Given a Schr\"odinger operator $H=-\Delta+V$ with bounded, measurable potential $V$ and some $E\in \RR$ we
call $\psi_n, n \in \NN,$ a \emph{Weyl sequence} for $H$ and $E$ 
if $\psi_n\in W^{2,2} (\RR^d)$, $\lVert \psi_n \rVert =1$,  and $\lVert (H-E)\psi_n \rVert <1/n $ 
for all $n \in \NN$. 
\begin{thm} \label{t:scale_free-Weyl}
For any $M >0$, $\delta \in (0,M/2]$, any $(M,\delta)$-equidistributed sequence 
$Z$,
any measurable and bounded $V\colon {\RR^d}\to \RR$, and any Weyl sequence for $H$ and $E$ we have
\begin{equation*}
\frac{1}{2} \left( \frac{\delta}{M}\right)^{K_1(1+M^{4/3} \lVert V-E \rVert_\infty^{2/3})}
\lVert \psi_n \rVert 
\leq \lVert W_{\delta, Z} \,\psi_n \rVert 
\leq \lVert \psi_n \rVert 
\end{equation*}
for $n \geq \sqrt{2} \delta M(\delta/M)^{-(1/2) K_1 (1+ M^{4/3}\lVert V-E\rVert_\infty^{2/3})}$, where $K_1$ is taken from Theorem~\ref{thm:mit-Rest}.
\end{thm}

\begin{proof}
Theorem ~\ref{thm:mit-Rest} with $V$ replaced by $V-E$ gives
\begin{align*}
\left( \frac{\delta}{M} \right)^{K_1 (1+M^{4/3} \lVert V-E\rVert_\infty^{2/3})}
 \lVert \psi_n \rVert^2
 &\leq
 \lVert W_{\delta, Z} \, \psi_n \rVert^2 + \delta^2 M^2 \lVert (-\Delta + V-E)\psi_n \rVert^2 \\
&\leq
 \lVert W_{\delta, Z} \, \psi_n \rVert^2 + (\delta M/n)^2  .
\end{align*}
For $n^2 \geq 2 \delta^2 M^2 {\delta}^{-K_1 (1+ M^{4/3}\lVert V-E\rVert_\infty^{2/3})}$
we have $(\delta M/n)^2  \leq (1/2) {(\delta/M)}^{K_1 (1+ M^{4/3}\lVert V-E\rVert_\infty^{2/3})}$.
This finishes the proof.
\end{proof}

In the analysis of random Schr\"odinger operators $H_\omega=H_0+V_\omega$ with $H_0=-\Delta+V_{\rm per}$ 
one often studies finite volume/finite box subsystems. This can be achieved by restricting 
$H_\omega$ itself to $\Lambda_L$ with Dirichlet or periodic boundary conditions. Alternatively, if one is interested only in energies in the resolvent set of $H_0$, the finite scale approximation $H_L:= H_0 +\chi_{\Lambda_L} V_\omega$ is possible as well, to which the Theorem above 
applies (under standard model assumptions). 

Meanwhile, Theorems 1, 4, and 5 have been extended in  \cite{BorisovTV}
to variable coefficient elliptic second oder operators. 
\small

\subsubsection*{Acknowledgment}
I.V. would like to thank J.-M.~Barbaroux,  G.~Raikov, and C.~Rojas-Molina for stimulating discussions.
This work has been partially funded by the DFG through grant \emph{Unique continuation principles and equidistribution properties of eigenfunctions}.

\end{document}